\documentclass{lmcs}
\pdfoutput=1

\usepackage{lastpage}
\lmcsdoi{18}{1}{27}
\lmcsheading{}{\pageref{LastPage}}{}{}%
{Dec.~22,~2020}{Feb.~03,~2022}{}

\usepackage{amsthm, amsmath, amsfonts, amssymb, graphicx, tikz-cd, float, hyperref}
\usepackage[utf8]{inputenc}
\usepackage[labelsep=none]{caption}

\newcommand\caba{\sf{CABA}}
\def\ba{{\sf{BA}}}
\def\cama{{\sf{CAMA}}}
\def\ma{{\sf{MA}}}
\def\SL{{\sf{SL}}}
\def\SSL{{\sf{StoneSL}}}
\def\Set{{\sf{Set}}}
\def\KFr{{\sf{KFr}}}
\def\DFr{{\sf{DFr}}}
\def\Stone{{\sf{Stone}}}
\def\Alg{{\sf{Alg}}}
\def\Coalg{{\sf{Coalg}}}
\def\csl{{\sf{CSL}}}
\def\op{{\sf{op}}}

\def\V{\mathcal{V}}
\def\H{\mathcal{H}}
\def\U{\mathcal{U}}
\def\C{\mathcal{C}}
\def\A{\mathcal{A}}
\def\M{\mathcal{M}}
\def\K{\mathcal{K}}
\def\F{\mathcal{F}}
\def\L{\mathcal{L}}
\def\T{\mathcal{T}}
\def\U{\mathcal{U}}
\def\At{{\mathfrak{at}}}
\def\uf{{\mathfrak{uf}}}
\def\Clop{{\mathfrak{clop}}}

\def\down{{\downarrow}}

\title{Duality for Powerset Coalgebras}

\author[G.~Bezhanishvili]{Guram~Bezhanishvili\rsuper{a}}
\address{New Mexico State University}
\email{\{guram,pmorandi\}@nmsu.edu}

\author[L.~Carai]{Luca~Carai\rsuper{b}}
\address{Universit\`a degli Studi di Salerno}
\email{lcarai@unisa.it}

\author[P.~J.~Morandi]{Patrick~J.~Morandi\rsuper{a}}


\begin{document}

\begin{abstract}
  Let $\caba$ be the category of complete atomic boolean algebras and complete boolean homomorphisms, and let $\csl$ be the category of complete meet-semilattices and complete meet-homomorphisms. We show that the forgetful functor from $\caba$ to $\csl$ has a left adjoint. This allows us to describe an endofunctor $\H$ on $\caba$ such that the category $\Alg(\H)$ of algebras for $\H$ is dually equivalent to the category $\Coalg(\mathcal{P})$ of coalgebras for the powerset endofunctor $\mathcal{P}$ on $\Set$. As a consequence, we derive Thomason duality from Tarski duality, thus paralleling how J\'onsson-Tarski duality was derived from Stone duality in~\cite{Abr88,KKV04}.
\end{abstract}

\maketitle

\section{Introduction}

It is a classic result in modal logic, known as \emph{J\'onsson-Tarski duality}, that the category $\ma$ of modal algebras is dually equivalent to the category $\DFr$ of descriptive frames. This result can be traced back to the work of J\'onsson-Tarski~\cite{JT51}, Halmos~\cite{Hal56}, and Kripke~\cite{Kri63}. In the modern form it was proved by Esakia~\cite{Esa74} and Goldblatt~\cite{Gol76}.\footnote{We point out that Esakia phrased it for the subcategory of descriptive frames where the relation $R$ is reflexive and transitive. Consequently, he worked with the subcategory of modal algebras consisting of closure algebras of McKinsey and Tarski~\cite{MT44}.}

J\'onsson-Tarski duality is a generalization of the celebrated Stone duality between the category $\ba$ of boolean algebras and the category $\Stone$ of Stone spaces. It was observed by Abramsky~\cite{Abr88} and Kupke, Kurz, and Venema~\cite{KKV04} that J\'onsson-Tarski duality can be proved by lifting Stone duality using algebra/coalgebra methods. This can be done by utilizing the classic Vietoris construction (see, e.g.,~\cite[Ch.~III.4]{Joh82}). Indeed, associating with each Stone space $X$ its Vietoris space $\V(X)$ gives rise to an endofunctor $\V:\Stone\to\Stone$ such that $\DFr$ is isomorphic to the category $\Coalg(\V)$ of coalgebras for $\V$. Let $\SL$ be the category of meet-semilattices with top. Then the forgetful functor $\U:\ba\to\SL$ has a left adjoint $\L:\SL\to\ba$. Letting $\K=\L\U$ gives an endofunctor $\K:\ba\to\ba$ such that $\ma$ is isomorphic to the category $\Alg(\K)$ of algebras for $\K$. Moreover, the following diagram commutes up to natural isomorphism, yielding that Stone duality lifts to a dual equivalence between $\Alg(\K)$ and $\Coalg(\V)$. This provides an alternate proof of J\'onsson-Tarski duality. 

\begin{center}
\begin{tikzcd}[column sep = 8pc]
\ba \arrow[r, <->, "\textrm{Stone duality}"] \arrow[d, "\K"'] &  \Stone \arrow[d, "\V"] \\
\ba  \arrow[r, <->, "\textrm{Stone duality}"'] & \Stone 
\end{tikzcd}
\end{center}

Descriptive frames can be thought of as Kripke frames $(X,R)$ equipped with a Stone topology compatible with the relation $R$.
In~\cite{Tho75} Thomason proved a ``discrete version" of J\'onsson-Tarski duality, establishing that the category $\KFr$ of Kripke frames is dually equivalent to the category $\cama$ of complete atomic modal algebras whose modal operator is completely multiplicative (see Sections~\ref{sec:free} and~\ref{sec:coalgebra} for all the undefined notions). 

The same way J\'onsson-Tarski duality generalizes Stone duality, Thomason duality generalizes Tarski duality between the category $\caba$ of complete atomic boolean algebras and the category $\Set$ of sets. It is natural to try to obtain Thomason duality from Tarski duality using algebra/coalgebra methods in the same vein J\'onsson-Tarski duality was obtained from Stone duality in~\cite{Abr88,KKV04}. Surprisingly, such an approach has not yet been undertaken. Our aim is to fill in this gap.

For this purpose, it is natural to replace the Vietoris endofunctor $\V$ on $\Stone$ with the powerset endofunctor $\mathcal{P}$ on $\Set$. It is known (see, e.g.,~\cite[Sec.~9]{Ven07}) that $\KFr$ is isomorphic to $\Coalg(\mathcal{P})$. Thus, the key is to construct an endofunctor $\H$ on $\caba$ that is an analogue of the endofunctor $\K:\ba\to\ba$. We recall that $\K=\L\U$ where $\L:\SL\to\ba$ is left adjoint to the forgetful functor $\U:\ba\to\SL$. A natural analogue of $\SL$ in the complete case is the category $\csl$ of complete meet-semilattices. 
Our main contribution is to show that the forgetful functor $\U:\caba\to\csl$ has a left adjoint $\L:\csl\to\caba$. We then define $\H:\caba\to\caba$ as the composition $\H=\L\U$, and prove that $\Alg(\H)$ is dually equivalent to $\Coalg(\mathcal P)$. Since $\cama$ is isomorphic to $\Alg(\H)$ and $\Coalg(\mathcal P)$ is isomorphic to $\KFr$, Thomason duality follows. 

The paper is organized as follows. In Section~\ref{sec:free} we recall J\'onsson-Tarski duality and how it can be obtained by lifting Stone duality. For this we need to work with the left adjoint $\L:\SL\to\ba$ of the forgetful functor $\U:\ba\to\SL$. 
The standard approach to constructing $\L$ is to take the free boolean algebra over the underlying set of a meet-semilattice with top and quotient it by the relations defining a modal operator $\Box$ (see~\cite[Prop.~3.12]{KKV04}). 
As we will show in Theorem~\ref{prop: H is a reflector}, $\L$ can alternatively be constructed by utilizing Pontryagin duality for semilattices~\cite{HMS74}. Let $\sf StoneSL$ be the category of topological meet-semilattices, where the topology is a Stone topology. Then $\SL$ is dually equivalent to $\sf StoneSL$ (see~\cite[Thm.~3.9]{HMS74} or~\cite[p.~251]{Joh82}). 
We will show that $\L$ can be constructed by taking the boolean algebra of clopen subsets of the Pontryagin dual $M^*:= \hom_\SL(M,2)$ of $M \in \SL$. This shows that $\L:\SL\to\ba$ can be constructed either purely algebraically, by utilizing the existence of free algebras in $\ba$, or using duality, as the boolean algebra of clopens of the Pontryagin dual of a meet-semilattice with top.

In Section~\ref{sec:coalgebra} we show that the same approach applies to the forgetful functor $\U:\caba\to\csl$. Its left adjoint $\L:\csl\to\caba$ can be constructed by taking the free object in $\caba$ over the underlying set of a complete meet-semilattice and then taking the quotient of it by the relations defining a completely multiplicative modal operator $\Box$. We point out that care is needed in constructing the free object in $\caba$ since it is well known that free objects do not exist in the category of complete boolean algebras (see~\cite{Gai64,Hal64}). Nevertheless, free objects in $\caba$ do exist. This can be seen by observing that the Eilenberg-Moore algebras of the double contravariant powerset monad are exactly the objects of $\caba$~\cite{Tay02}, and that categories of algebras for monads have free objects~\cite[Prop.~20.7(2)]{AHS06}. A more concrete construction of free objects in $\caba$ can be given by utilizing the theory of canonical extensions of J\'onsson and Tarski~\cite{JT51}. Indeed, in Theorem~\ref{prop: free caba} we will prove that the free object in $\caba$ over a set $X$ is the canonical extension $F^\sigma$ of the free boolean algebra $F$ over $X$. We then quotient $F^\sigma$ by the complete congruence generated by the relations defining a completely multiplicative modal operator, yielding the desired $\L:\csl\to\caba$ (see Theorem~\ref{prop: H is a reflector Venema def}).

An alternate construction of $\L:\csl\to\caba$ that parallels the alternate construction of $\L:\SL\to\ba$ can be given by taking the powerset of $\hom_\csl(M,2)$ for each $M\in\csl$. Since $\hom_\csl(M,2)$ is isomorphic to the order-dual of $M$, this amounts to taking the powerset of $M$.\footnote{We thank one of the referees for suggesting this approach.} Therefore, we again obtain that the left adjoint can be constructed either purely algebraically, utilizing that free objects exist in $\caba$, or else using the powerset construction. Thus, we arrive at the following diagram, which parallels the constructions of the left adjoints $\L:\SL\to\ba$ and $\L:\csl\to\caba$:

\medskip

\begin{center}
\begin{tabular}{|l|l|l|} \hline
\begin{tabular}{l}Forgetful functor\end{tabular} & \begin{tabular}{l}Left adjoint\end{tabular} & \begin{tabular}{l}Location\end{tabular} \\ \hline
\begin{tabular}{l}$\U : \ba \to \SL$ \\  \phantom{algebraic} \\  \phantom{dual} \end{tabular}
& 
\begin{tabular}{l}$\mathcal{L} : \SL \to \ba$ \\ algebraic construction:  ${\sf Free}_{\ba}(M)/{\equiv}$ \\ dual construction: $\Clop(\hom_{\SL}(M,2))$ \end{tabular} 
& 
\begin{tabular}{l} \\ \cite{KKV04} \\ Theorem~\ref{prop: H is a reflector} \end{tabular}
\\ \hline
\begin{tabular}{l}$\U : \caba \to \csl$ \\ \phantom{algebraic} \\ \phantom{dual} \end{tabular}
& 
\begin{tabular}{l}$\mathcal{L} : \csl \to \caba$  \\ algebraic construction:  ${\sf Free}_{\caba}(M)/{\equiv}$ \\ dual construction: $\wp(\hom_{\csl}(M,2)) \cong \wp(M)$ \end{tabular}
& 
\begin{tabular}{l} \ \\ Theorem~\ref{prop: H is a reflector Venema def} \\ Theorem~\ref{thm:alternate L} \end{tabular}
\\ \hline
\end{tabular}
\end{center}

\medskip

In Section~\ref{sec:thomason} we define the endofunctor $\H:\caba\to\caba$ as the composition $\H=\L\U$. In Theorem~\ref{thm:Hwp=wpP} we prove that the following diagram commutes (up to natural isomorphism). 

\begin{center}
\begin{tikzcd}[column sep = 8pc]
\caba \arrow[r, <->, "\textrm{Tarski duality}"] \arrow[d, "\H"'] &  \Set \arrow[d, "\mathcal P"] \\
\caba  \arrow[r, <->, "\textrm{Tarski duality}"'] & \Set 
\end{tikzcd}
\end{center}
This paves the way towards proving that the category $\Alg(\H)$ of algebras for $\H$ is dually equivalent to the category $\Coalg(\mathcal P)$ of coalgebras for $\mathcal P$ (see Theorem~\ref{cor: main}).

It is well known that 
J\'onsson-Tarski and Thomason dualities are connected through the canonical extension and forgetful functors $(-)^\sigma:\ma\to\cama$ and $\U:\DFr\to\KFr$, making the following diagram commutative.
\begin{center}
\begin{tikzcd}[column sep = 8pc]
\ma \arrow[r, <->, "\textrm{J\'onsson-Tarski duality}"] \arrow[d, "(-)^\sigma"'] &  \DFr \arrow[d, "\U"] \\
\cama  \arrow[r, <->, "\textrm{Thomason duality}"'] & \KFr 
\end{tikzcd}
\end{center}
We conclude the paper by Remark~\ref{rem: canonical ext on algebras} in which we show that there are analogous canonical extension and forgetful functors $(-)^\sigma:\Alg(\K)\to\Alg(\H)$ and $\U:\Coalg(\V) \to\Coalg(\mathcal{P})$ that make a similar diagram commutative.

\section{Coalgebraic approach to J\'onsson-Tarski duality}\label{sec:free}

In this section we give a brief account of J\'onsson-Tarski duality, and provide a construction of the left adjoint $\L:\SL\to\ba$ of the forgetful functor $\U:\ba\to\SL$ which is alternative to~\cite[Prop.~3.12]{KKV04}. This we do by utilizing Pontryagin duality for semilattices~\cite{HMS74}. We start by recalling the definition of a modal algebra.

\begin{defi}
A \emph{modal algebra} is a pair $(B,\Box)$ where $B$ is a boolean algebra and $\Box$ is a unary function on $B$ preserving finite meets. A \emph{modal algebra homomorphism} between modal algebras $(B_1,\Box_1)$ and $(B_2,\Box_2)$ is a boolean homomorphism $\alpha:B_1\to B_2$ such that $\alpha(\Box_1 a)=\Box_2\alpha(a)$ for each $a\in B_1$. Let $\ma$ be the category of modal algebras and modal algebra homomorphisms. 
\end{defi}

A subset of a topological space $X$ is {\em clopen} if it is both closed and open, and $X$ is {\em zero-dimensional} if $X$ has a basis of clopen sets. A {\em Stone space} is a zero-dimensional compact Hausdorff space. 

For a binary relation $R$ on $X$, we write 
\[
R[x]:=\{y \in X \mid xRy\} \ \mbox{ and } \ R^{-1}[U]:=\{x\in X \mid \exists u\in U \mbox{ with } xRu\}
\] 
for the $R$-image of $x\in X$ and $R$-inverse image of $U\subseteq X$.

\begin{defi}
A \emph{descriptive frame} is a pair $(X,R)$ where $X$ is a Stone space and $R$ is a binary relation on $X$ such that $R[x]$ is closed for each $x\in X$ and $R^{-1}[U]$ is clopen for each clopen $U \subseteq X$.
\end{defi}

Such relations are often called \emph{continuous relations} for the following reason. Let $\V(X)$ be the \emph{Vietoris space} of $X$. We recall (see, e.g.,~\cite[Sec.~III.4]{Joh82}) that $\V(X)$ is the set of closed subsets of $X$ topologized by the subbasis $\{\Box_U,\Diamond_V\mid U,V \mbox{ open in } X\}$ where
\[
\Box_U=\{ F\in\V(X) \mid F \subseteq U\} \ \mbox{ and } \ \Diamond_V=\{F\in\V(X) \mid F\cap V\ne\varnothing\}.
\] 
Then $R$ is continuous iff the associated map $\rho_R:X\to\V(X)$, given by $\rho_R(x)=R[x]$, is a well-defined continuous map (that $\rho_R$ is well defined follows from (i), and that it is continuous from (ii)). 

Let $\DFr$ be the category of descriptive frames and continuous p-morphisms, where a \textit{p-morphism} between $(X_1,R_1)$ and $(X_2,R_2)$ is a map $f:X_1 \to X_2$ satisfying $f[R_1[x]]=R_2[f(x)]$ for each $x\in X_1$. 

\begin{thm} [J\'onsson-Tarski duality]
$\ma$ is dually equivalent to $\DFr$.
\end{thm}

J\'onsson-Tarski duality generalizes Stone duality between the category $\ba$ of boolean algebras and boolean homomorphisms and the category $\Stone$ of Stone spaces and continuous maps. We recall that the contravariant functors $\uf:\ba\to^\op\Stone$ and $\Clop:\Stone\to^\op\ba$ yielding Stone duality are constructed as follows.\footnote{To easily distinguish between covariant and contravariant functors, following the suggestion of one of the referees, we write $F:{\sf C}\to^\op{\sf D}$ for a contravariant functor $F$.}
The functor $\uf$ assigns to each boolean algebra $A$ the set $\uf(A)$ of ultrafilters of $A$ topologized by the basis $\{\beta_A(a)\mid a\in A\}$ where 
\[
\beta_A(a)=\{x\in\uf(A)\mid a\in x\}.
\] 
To each boolean homomorphism $\alpha:A\to B$, the functor $\uf$ assigns $\uf(\alpha):=\alpha^{-1}:\uf(B)\to\uf(A)$. 
The functor $\Clop$ assigns to each Stone space $X$ the boolean algebra $\Clop(X)$ of clopen subsets of $X$, and to each continuous map $f:X\to Y$ the boolean homomorphism $\Clop(f):=f^{-1}:\Clop(Y)\to\Clop(X)$. 

One unit $\beta : 1_{\ba} \to \Clop \circ \uf$ of this dual equivalence is given by the Stone maps $\beta_A:A\to\Clop(\uf(A))$ for $A\in\ba$, and the other unit $\eta : 1_{\Stone} \to \uf\circ\Clop$ by the homeomorphisms $\eta_X:X\to\uf(\Clop(X))$ for $X\in\Stone$, which are given by 
\[
\eta_X(x)=\{U\in\Clop(X)\mid x\in U\}.
\] 

These functors naturally generalize to yield J\'onsson-Tarski duality. As we pointed out in the introduction, an alternative approach to J\'onsson-Tarski duality is by lifting Stone duality using algebra/coalgebra methods. We recall that the Vietoris construction extends to an endofunctor $\V:\Stone\to\Stone$ by sending a continuous map $f:X\to Y$ to $\V(f):\V(X)\to\V(Y)$ given by $\V(f)(G)=f[G]$ for each $G\in\V(X)$. We next consider the category $\Coalg(\V)$ of coalgebras for $\V$. For this we recall the notion of a coalgebra for an endofunctor (see, e.g.,~\cite[Def.~9.1]{Ven07}).

\begin{defi}\label{def:coalgebra}
\begin{enumerate}
\item[]
\item A {\em coalgebra} for an endofunctor $\T:\sf{C} \to \sf{C}$ is a pair $(A,f)$ where $A$ is an object of the category $\sf{C}$ and $f:A \to \T(A)$ is a $\sf{C}$-morphism.
\item A {\em morphism} between two coalgebras $(A_1, f_1)$ and $(A_2,f_2)$ for $\T$ is a $\sf{C}$-morphism $\alpha : A_1 \to A_2$ such that the following square is commutative.
\[
\begin{tikzcd}[column sep = 5pc]
A_1 \arrow[d, "f_1"'] \arrow[r, "\alpha"] &  A_2 \arrow[d, "f_2"] \\
\T(A_1) \arrow[r, "\T(\alpha)"'] &\T(A_2)
\end{tikzcd}
\]
\item Let $\Coalg(\T)$ be the category whose objects are coalgebras for $\T$ and whose morphisms are morphisms of coalgebras.
\end{enumerate}
\end{defi}

The dual endofunctor $\K:\ba\to\ba$ of the Vietoris endofunctor $\V:\Stone\to\Stone$ was described in~\cite{KKV04}. Let $\SL$ be the category of meet-semilattices with top and meet-homomorphisms preserving top. Then $\K$ is the composition $\L\U$, where $\L:\SL\to\ba$ is the left adjoint of the forgetful functor $\U:\ba\to\SL$. In~\cite[Prop.~3.12]{KKV04} the left adjoint is constructed algebraically, by taking the free boolean algebra $F$ over the underlying set of $M\in\SL$ and then taking the quotient of $F$ by the relations remembering that $M$ is a meet-semilattice with top. We give an alternative description of $\L$, which utilizes Pontryagin duality for semilattices~\cite{HMS74}, which we briefly recall next. 

Let $\SSL$ be the category whose objects are topological meet-semilattices, where the topology is a Stone topology, and whose morphisms are continuous meet-homomorphisms. Pontryagin duality for $\SL$ establishes a dual equivalence between $\SL$ and $\SSL$.
The contravariant functor $(-)^* : \SL \to^\op \SSL$ sends $M$ to its dual $M^* := \hom_{\SL}(M, 2)$, where $2 = \{0,1\}$ is the two-element chain and meet on $M^*$ is pointwise meet. If $2$ is given the discrete topology and $2^M$ the product topology, then $M^*$ is easily seen to be a closed subspace of $2^M$, and so the subspace topology is a Stone topology. Moreover, pointwise meet is continuous, and hence $M^*\in\SSL$. On morphisms, if $\sigma : M \to N$ is an $\SL$-morphism, then $\sigma^* : N^* \to M^*$ is defined by $\sigma^*(\gamma) = \gamma \circ \sigma$. The contravariant functor in the other direction sends $A \in \SSL$ to $A^*:=\hom_{\SSL}(A, 2)$ and $\sigma : A \to B$ to $\sigma^* : B^* \to A^*$, defined in the same way as the previous functor. Finally, one natural isomorphism sends each $M \in \SL$ to its double dual $M^{**}$ by sending $m$ to the map $\sigma \mapsto \sigma(m)$ for each $m \in M$ and $\sigma \in M^*$. The other natural isomorphism sends each $A \in \SSL$ to its double dual $A^{**}$ and is given by the same formula. 

\begin{thm} \label{prop: H is a reflector}
Associating with each $M \in \SL$ the boolean algebra $\Clop(M^*)$ of clopen subsets of its dual $M^*$ yields an alternative description of the functor $\mathcal{L}:\SL \to \ba$ that is left adjoint to the forgetful functor $\U:\ba \to \SL$.
\end{thm}

\begin{proof}
Let $M\in\SL$. Define $i_M : M \to \Clop(M^*)$ by $i_M(m) = \{ \sigma \in M^* \mid \sigma(m) = 1\}$. It is straightforward to see that $i_M$ is a well-defined $\SL$-morphism. By~\cite[p.~89]{Mac71} it is enough to show that for each $A\in\ba$ and $\SL$-morphism $\gamma : M \to A$ there is a unique $\ba$-morphism $\tau : \Clop(M^*) \to A$ such that $\tau\circ i_M = \gamma$. 
\[
\begin{tikzcd}[column sep = 5pc]
M \arrow[r, "i_M"] \arrow[dr, "\gamma"'] & \Clop(M^*) \arrow[d, "\tau"] \\
& A
\end{tikzcd}
\]
The map $\gamma^* : \hom_{\SL}(A, 2) \to M^*$ is continuous, so its restriction $\gamma^* : \hom_{\ba}(A, 2) \to M^*$ is continuous. Therefore, $\Clop(\gamma^*) : \Clop(M^*) \to \Clop(\hom_{\ba}(A, 2))$ is a $\ba$-morphism. If $m \in M$, then
\begin{align*}
\Clop(\gamma^*)(i_M(m)) &= (\gamma^*)^{-1}(i_M(m)) = \{ \sigma \in M^* \mid \gamma^*(\sigma) \in i_M(m) \} \\
&= \{ \sigma \in M^* \mid \sigma(\gamma(m)) = 1 \}. %%%%% broke this into two lines
\end{align*}
Let $\tau : \Clop(M^*) \to A$ be the composition of $\Clop(\gamma^*)$ with the inverse of the natural isomorphism $\beta_A: A \to \Clop(\hom_{\ba}(A, 2))$ of Stone duality which sends $a\in A$ to $\{ \sigma\in~\hom_{\ba}(A, 2) \mid \sigma(a) = 1\}$.\footnote{Here we make the well-known identification of $\uf(A)$ with $\hom_{\ba}(A, 2)$.} %%%%% added tilde 
Then 
\[
\tau(i_M(m)) = \beta_A^{-1} \Clop(\gamma^*) (i_M(m))= \beta_A^{-1}(\{ \sigma \in M^* \mid \sigma(\gamma(m)) = 1 \})= \gamma(m),
\] 
so $\tau \circ i_M = \gamma$. Finally, uniqueness of $\tau$ follows since $i_M[M]$ generates $\Clop(M^*)$ as a boolean algebra.
\end{proof}

\begin{rem}
To see how $\L$ acts on morphisms, if $\sigma : M \to N$ is an $\SL$-morphism, then $\sigma^* : N^* \to M^*$ is a continuous $\SL$-morphism between Stone spaces, so $\Clop(\sigma^*) : \Clop(M^*) \to \Clop(N^*)$ is a $\ba$-morphism by Stone duality.  We then set $\L(\sigma) = \Clop(\sigma^*) :\L(M) \to \L(N)$.
\end{rem}

Let $\K=\L\U$. Then $\K$ is an endofunctor on $\ba$.
\[
\begin{tikzcd}[column sep = 5pc]
\ba \arrow[rr, bend left = 15, "\K"] \arrow[r, "\U"'] & \SL \arrow[r, "\L"'] & \ba
\end{tikzcd}
\]
Let $\Alg(\K)$ be the category of algebras for $\K$ (see~\cite[Def.~5.37]{AHS06}). We recall that for an endofunctor $\T$, algebras for $\T$ are defined by reversing the arrows in the definition of coalgebras for $\T$. Since $\K$ is dual to $\V$, we have that $\Alg(\K)$ is dually equivalent to $\Coalg(\V)$. Because $\Alg(\K)$ is isomorphic to $\ma$ and $\Coalg(\V)$ is isomorphic to $\DFr$, this gives an alternate proof of J\'onsson-Tarski duality (see~\cite{KKV04}).

\section{Two constructions of the left adjoint \texorpdfstring{$\L : \csl \to \caba$}{L : CSL --> CABA}}\label{sec:coalgebra}

Let $\KFr$ be the category of Kripke frames and p-morphisms. Forgetting the topology of a descriptive frame yields the forgetful functor $\U:\DFr\to\KFr$. To describe the modal algebras corresponding to Kripke frames, we recall the notion of a completely multiplicative modal operator.

\begin{defi}
A modal operator $\Box$ on a complete boolean algebra $B$ is \emph{completely multiplicative} if $\Box\left(\bigwedge S\right)=\bigwedge\{\Box s\mid s\in S\}$ for each $S\subseteq B$. Let $\cama$ be the category whose objects are complete atomic modal algebras with completely multiplicative $\Box$, and whose morphisms are complete modal algebra homomorphisms. 
\end{defi}

\begin{thm} [Thomason duality] 
$\cama$ is dually equivalent to $\KFr$. 
\end{thm}

Thomason duality generalizes Tarski duality between $\caba$ and $\Set$ the same way J\'onsson-Tarski duality generalizes Stone duality. We recall that $\caba$ is the category of complete atomic boolean algebras and complete boolean homomorphisms and $\Set$ is the category of sets and functions.  The contravariant functors of Tarski duality are $\wp:\Set\to^\op\caba$ and $\At:\caba\to^\op\Set$. The functor $\wp$ assigns to each set $X$ the powerset $\wp(X)$ and to each function $f:X\to Y$ its inverse image $f^{-1}:\wp(Y)\to\wp(X)$. The functor $\At$ assigns to each $A\in\caba$ its set of atoms. If $\alpha:A\to B$ is a complete boolean homomorphism, it has a left adjoint $\alpha^*:B\to A$, which sends atoms to atoms, and the functor $\At$ assigns to $\alpha$ the function $\alpha^*:\At(B)\to\At(A)$. One unit $\varepsilon : 1_{\Set} \to \At\circ\wp$ of this dual equivalence is given by $\varepsilon_X(x)=\{x\}$ for each $x \in X \in \Set$, and the other unit $\vartheta : 1_{\caba} \to \wp \circ \At$ by $\vartheta_A(a)={\downarrow}a\cap \At(A)$ for each $a\in A \in \caba$.

To derive Thomason duality from Tarski duality the same way J\'onsson-Tarski duality was derived from Stone duality, we need to replace the Vietoris endofunctor $\V$ on $\Stone$ with the powerset endofunctor $\mathcal P$ on $\Set$. We recall that the endofunctor $\mathcal{P}:\Set \to \Set$ associates to each set $X$ its powerset $\mathcal{P}(X)$ and to each function $f:X \to Y$ the function $\mathcal{P}(f):\mathcal{P}(X) \to \mathcal{P}(Y)$ that maps each subset $S \subseteq X$ to its direct image $f[S]$. We also need to replace the endofunctor $\K:\ba\to\ba$ with an appropriate endofunctor $\H:\caba\to\caba$. 

To describe $\H$, we need to construct the left adjoint to $\U:\caba\to\csl$, where $\csl$ is the category of complete meet-semilattices and complete meet-homomorphisms. As in the previous section, this can be done purely algebraically or using duality. To construct $\H$ algebraically, we need that free objects exist in $\caba$. Care is needed here since
it is a well-known result of Gaifman~\cite{Gai64} and Hales~\cite{Hal64} that free objects do not exist in the category of complete boolean algebras and complete boolean homomorphisms. On the other hand, free objects do exist in $\caba$, and this can be seen by observing that the Eilenberg-Moore algebras of the double contravariant powerset monad are exactly the objects of $\caba$~\cite{Tay02}, and that categories of algebras for monads have free objects~\cite[Prop.~20.7(2)]{AHS06}. 

A more concrete construction of free objects in $\caba$ can be given utilizing the theory of canonical extensions. It is well known that free objects over any set exist in the category of complete and completely distributive lattices (see Markowski~\cite{Mar79} and Dwinger~\cite[Thm.~4.2]{Dwi81}). By~\cite[Cor.~2.3]{BHJ20}, the free complete and completely distributive lattice over a set $X$ is the canonical extension of the free bounded distributive lattice over $X$. We show that the same is true in $\caba$.
For this we need to recall the definition of a canonical extension of a boolean algebra.

\begin{defi}~\cite{JT51,GH01} 
A {\em canonical extension} of a boolean algebra $A$ is a complete boolean algebra $A^\sigma$ together with a boolean embedding $e:A\to A^\sigma$ satisfying:  
\begin{enumerate}
\item (Density) Each $x\in A^\sigma$ is a join of meets (and hence also a meet of joins) of $e[A]$. 
\item (Compactness) For $S,T\subseteq A$, from $\bigwedge e[S] \le \bigvee e[T]$ it follows that $\bigwedge S_0 \le \bigvee T_0$ for some finite $S_0\subseteq S$ and $T_0\subseteq T$. 
\end{enumerate}
\end{defi}

It is well known that canonical extensions are unique up to isomorphism, and that the correspondence $A\mapsto A^\sigma$ extends to a covariant functor $(-)^\sigma:\ba\to\caba$. It can conveniently be described as the composition $\wp\circ\U\circ\uf$, where $\U:\Stone\to\Set$ is the forgetful functor. 
\[
\begin{tikzcd}
\ba \arrow[rrr, bend left = 20, "(-)^\sigma"] \arrow[r, "\uf"'] & \Stone \arrow[r, "\U"'] & \Set \arrow[r, "\wp"'] & \caba
\end{tikzcd}
\]
Thus, we can think of $A^\sigma$ as $\wp(\uf(A))$ and of $e:A\to A^\sigma$ as the Stone map $\beta_A:A\to\wp(\uf(A))$. 

\begin{thm} \label{prop: free caba}
Let $X$ be a set. The canonical extension of the free boolean algebra over $X$ is the free object in $\caba$ over $X$.
\end{thm}

\begin{proof}
Let $F$ be the free boolean algebra over $X$, $f : X \to F$ the associated map, and $e : F \to F^\sigma$ the boolean embedding into the canonical extension. We show that $(F^\sigma, e\circ f)$ has the universal mapping property in $\caba$. Let $A \in \caba$ and $g : X \to A$ be a function. Since $A$ is a boolean algebra, there is a unique boolean homomorphism $\varphi : F \to A$ with $\varphi \circ f = g$. This induces a map $\uf(\varphi) : \uf(A) \to \uf(F)$  
given by $\uf(\varphi)(y)=\varphi^{-1}(y)$. Define $\varphi_+ : \At(A)\to \uf(F)$ by $\varphi_+(x)=\varphi^{-1}({\uparrow}x)$. If we identify atoms with the principal ultrafilters, we can think of $\varphi_+$ as the restriction of $\uf(\varphi)$ to $\At(A)$.

We identify $F^\sigma$ with $\wp(\uf(F))$. Then $e:F\to F^\sigma$ becomes the Stone map $\beta_F$.
The map $\varphi_+:\At(A)\to \uf(F)$ yields a $\caba$-morphism $\wp(\varphi_+) : F^\sigma \to \wp(\At(A))$. Since $A \in \caba$, the map $\vartheta_A : A \to \wp(\At(A))$ is an isomorphism. We set $\psi = \vartheta_A^{-1} \circ \wp(\varphi_+)$. Clearly $\psi  : F^\sigma \to A$ is a $\caba$-morphism.
We show that $\vartheta_A \circ \varphi = \wp(\varphi_+) \circ e$.
\[
\begin{tikzcd}[column sep = 5pc]
& F \arrow[r, "e"] \arrow[d, "\varphi"] & F^\sigma \arrow[d, "\wp(\varphi_+)"] \arrow[dl, shift right = .3pc, "\psi"] \\
X \arrow[ur, "f"] \arrow[r, "g"'] &A \arrow[r, "\vartheta_A"'] & \wp(\At(A))
\end{tikzcd}
\]
Let $a \in F$. Since $\vartheta_A\varphi(a) = \{ x \in \At(A) \mid x \le \varphi(a) \}$ and $e(a) = \beta_F(a) = \{ y \in \uf(F) \mid a \in y\}$, we have
\begin{align*}
(\wp(\varphi_+)\circ e)(a) &= \varphi_+^{-1}e(a) = \{ x \in \At(A) \mid \varphi_+(x) \in e(a) \} \\
&= \{ x \in \At(A) \mid a \in \varphi_+(x) \} = \{ x \in \At(A) \mid a \in \varphi^{-1}({\uparrow}x) \} \\
&= \{ x \in \At(A) \mid x \le \varphi(a)\} = \vartheta_A \varphi(a).
\end{align*}
This shows that $\vartheta_A \circ \varphi = \wp(\varphi_+) \circ e$, so 
\[
\psi \circ (e \circ f) =  \vartheta_A^{-1} \circ \wp(\varphi_+) \circ e \circ f =  \vartheta_A^{-1} \circ \vartheta_A \circ \varphi \circ f =  \varphi \circ f = g.
\]

It is left to show uniqueness. Suppose that $\mu : F^\sigma \to A$ is a $\caba$-morphism satisfying $\mu \circ (e \circ f) = g$. Then $(\mu \circ e) \circ f = (\psi \circ e) \circ f = \varphi \circ f$. By uniqueness of $\varphi$, we have $ \mu \circ e = \varphi = \psi \circ e$. Therefore, $\mu$ and $\psi$ agree on $e[F]$. Since $e[F]$ is dense in $F^\sigma$ and $\mu, \psi$ are $\caba$-morphisms, we conclude that $\mu = \psi$.
\end{proof}

We next show that the forgetful functor $\U:\caba\to\csl$ has a left adjoint $\L:\csl\to\caba$. 
Let $A\in\caba$. We recall that a boolean congruence $\equiv$ on $A$ is a \emph{complete congruence} if $a_i\equiv b_i$ for each $i\in I$ imply $\bigwedge\{ a_i\mid i\in I\} \equiv \bigwedge\{ b_i \mid i\in I\}$. It is well known that the quotient algebra $A/{\equiv}$ is also an object in $\caba$. As usual, for $a\in A$ we write $[a]$ for the equivalence class of $a$. Then the quotient map $\pi:A\to A/{\equiv}$, given by $a \mapsto [a]$, is a $\caba$-morphism.

\begin{rem} \label{rem:comp cong}
There is a well-known one-to-one correspondence between congruences and ideals of a boolean algebra $A$, which associates to each boolean congruence $\equiv$ on $A$ the equivalence class of $0$. If $A \in \caba$, this correspondence restricts to a one-to-one correspondence between complete congruences and principal ideals. In this case, the equivalence class of $0$ is generated by the element $x=\bigvee\{a\vartriangle b\mid a\equiv b\}$, where $\vartriangle$ denotes symmetric difference in $A$. 
\end{rem}

If $M \in \csl$, let $F(M)$ be the free object in $\caba$ over the underlying set of $M$, and let $f_M : M  \to F(M)$ be the associated map. We let $\equiv$ be the complete congruence on $F(M)$ generated by the relations:
\[
f_M\left(\bigwedge S\right) \equiv \bigwedge\{ f_M(s) \mid s\in S \}, \mbox{ where } S\subseteq M. 
\]
We then set $\L(M)$ to be the quotient algebra $F(M)/{\equiv}$. Since $F(M)\in\caba$ and $\equiv$ is a complete congruence, $\L(M)\in\caba$. For $a\in M$, let $\Box_a = [f_M(a)] \in \L(M)$. Let $\alpha_M : M \to \L(M)$ be the composition of the quotient map $\pi : F(M) \to \L(M)$ and $f_M$. Then $\alpha_M(a) = \Box_a$ for each $a \in M$.
\[
\begin{tikzcd}[column sep = 5pc]
M \arrow[r, "f_M"] \arrow[dr, "\alpha_M"'] & F(M) \arrow[d, "\pi"] \\
& \L(M)
\end{tikzcd}
\]
By the definition of $\equiv$ we see that $\Box_{\bigwedge S} = \bigwedge \{ \Box_s  \mid s \in S\}$ in $\L(M)$ for each $S \subseteq M$. Thus, $\alpha_M$ is a $\csl$-morphism.

\begin{rem} \label{rem: kernel}
In view of Remark~\ref{rem:comp cong}, the equivalence class $[0]\in\L(M)$ is the principal ideal generated by the element
\[
\bigvee \left\{ f_M\left({\bigwedge S}\right) \vartriangle \bigwedge \{ f_M(s) \mid s \in S\} \;\middle|\; S \subseteq M\right\}.
\]
\end{rem}

\begin{thm} \label{prop: H is a reflector Venema def}
The correspondence $M \mapsto \L(M)$ defines a functor $\L : \csl \to \caba$ that is left adjoint to the forgetful functor $\U:\caba \to \csl$.
\end{thm}

\begin{proof}
By~\cite[p.~89]{Mac71} it is enough to show that for each $M \in \csl$, $A\in\caba$, and $\csl$-morphism $\gamma : M \to A$ there is a unique $\caba$-morphism $\tau : \L(M) \to A$ such that $\tau\circ\alpha_M = \gamma$. There is a unique $\caba$-morphism $\varphi : F(M) \to A$ with $\varphi \circ f_M = \gamma$. To see that $\varphi$ factors through $\equiv$, let $S \subseteq M$. Since $\gamma$ is a $\csl$-morphism, $\gamma(\bigwedge S) = \bigwedge \{ \gamma(s) \mid s \in S\}$. Therefore,
\[
\varphi f_M\left(\bigwedge S \right) = \gamma \left(\bigwedge S  \right) = \bigwedge \gamma[S]
\]
and
\[
\varphi\left(\bigwedge \{ f_M(s) \mid s \in S\}\right) = \bigwedge \{ \varphi f_M(s) \mid s \in S\} = \bigwedge \{ \gamma(s) \mid s \in S\} = \bigwedge \gamma[S].
\]
\[
\begin{tikzcd}[column sep = 5pc]
M \arrow[r, "f_M"'] \arrow[dr, "\gamma"'] \arrow[rr, bend left = 20, "\alpha_M"] & F(M) \arrow[r, "\pi"'] \arrow[d, "\varphi"] & \L(M) \arrow[dl, "\tau"] \\
& A&
\end{tikzcd} %%%%% Is this the right location for the end of proof symbol?
\]
Thus, $\varphi f_M\left(\bigwedge S \right) = \varphi\left(\bigwedge \{ f_M(s) \mid s \in S\}\right)$.
This implies that $\equiv$ is contained in $\ker(\varphi)$, and hence $\varphi$ induces a $\caba$-morphism $\tau : \L(M) \to A$ with $\tau \circ \alpha_M = \gamma$. Since $\L(M)$ is generated by $\alpha_M[M]$ and $\tau$ is a $\caba$-morphism, $\tau$ is uniquely determined by the equation $\tau \circ \alpha_M = \gamma$.
\end{proof}

\begin{rem} \label{rem: H on maps}
To describe how $\L$ acts on morphisms, let $\gamma : M \to N$ be a $\csl$-morphism. Then $\alpha_N \circ \gamma : M \to \L(N)$ is a $\csl$-morphism, so there is a unique $\caba$-morphism $\L(\gamma) : \L(M) \to \L(N)$ such that $\L(\gamma) \circ \alpha_M = \alpha_N \circ \gamma$. 
\[
\begin{tikzcd}[column sep = 5pc]
M \arrow[r, "\gamma"] \arrow[d, "\alpha_M"'] & N \arrow[d, "\alpha_N"] \\
\L(M) \arrow[r, "\L(\gamma)"'] & \L(N)
\end{tikzcd}
\]
Therefore, if $a \in M$, then $\L(\gamma)(\Box_a) = \L(\gamma)\alpha_M(a) = \alpha_N\gamma(a) = \Box_{\gamma(a)}$.
\end{rem}

We conclude this section by giving an alternative construction of the left adjoint of the forgetful functor $\caba\to\csl$. 
In parallel to Theorem~\ref{prop: H is a reflector}, we can replace $\SL$ by $\csl$ and Stone duality by Tarski duality. Then for $M \in \csl$ we can consider $\L(M)$ to be $\wp(\hom_{\csl}(M, 2))$. If $\sigma \in \hom_{\csl}(M, 2)$, then $\sigma^{-1}(1)$ is a filter of $M$. Since $\sigma$ preserves arbitrary meets, letting $a = \bigwedge \sigma^{-1}(1)$ 
yields $\sigma^{-1}(1) = {\uparrow}a$. Conversely, if $a \in M$, then defining $\sigma_a$ by 
\[
\sigma_a(m) = \left\{\begin{array}{ll}1 & \textrm{if }a \le m\\ 0 & \textrm{otherwise}\end{array}\right.
\]
yields $\sigma_a \in \hom_{\csl}(M, 2)$. 
Because this correspondence reverses the order, there is an order-reversing bijection $f  : M \to \hom_{\csl}(M, 2)$ sending $a$ to $\sigma_a$. 
Thus, $\wp(f) : \wp(\hom_{\csl}(M, 2)) \to~\wp(M)$ %%%%% added tilde
is a $\caba$-isomorphism. If $i_M : M \to \wp(\hom_{\csl}(M,2))$ is given by $i_M(m) = \{ \sigma \in \hom_{\csl}(M,2) \mid \sigma(m) = 1\}$, then we have the function $\iota_M = \wp(f) \circ i_M : M \to \wp(M)$ given by
\begin{align*}
\iota_M(m) &= \wp(f)(i_M(m)) = \{ a \in M \mid f(a) \in i_M(m) \} = \{ a \in M \mid f(a)(m) = 1 \} \\
&= \{ a \in M \mid a \le m\} = {\downarrow}m .
\end{align*}
Thus, we may set $\L(M) = \wp(M)$ and define $\iota_M : M \to \L(M)$ by $\iota(m) = {\downarrow}m$.

\begin{thm}\label{thm:alternate L}
Associating with each $M \in \csl$ its powerset yields an alternative description of the functor $\mathcal{L}:\csl \to \caba$ that is left adjoint to the forgetful functor $\U:\caba \to \csl$.
\end{thm}

\begin{proof}
Let $M\in\csl$. Clearly the powerset of $M$ is an object in $\caba$. Define $\iota_M:M\to\L(M)$ by $\iota_M(a)=\down a$ for each $a \in M$. For each $S \subseteq M$, we have 
\[
  \bigwedge \{ \iota_M(s) \mid s \in S \} = \bigcap \{ \down s \mid s \in S \} = \big\downarrow\! \left(\bigwedge S\right) = \iota_M \left( \bigwedge S \right).
\]
Therefore, $\iota_M$ is a $\csl$-morphism. It is enough to show that for each 
$A\in\caba$ and $\csl$-morphism $\gamma : M \to A$ there is a unique $\caba$-morphism $\tau : \L(M) \to A$ such that $\tau\circ\iota_M = \gamma$. 
\[
\begin{tikzcd}[column sep = 5pc]
M \arrow[r, "\iota_M"] \arrow[dr, "\gamma"'] & \L(M) \arrow[d, "\tau"] \\
& A
\end{tikzcd}
\]

Let $\gamma^*$ be the left adjoint of $\gamma$, and consider its restriction $\gamma^*:\At(A) \to M$. Also recall that $\vartheta_A:A\to\wp(\At(A))$ is a $\caba$-isomorphism, hence so is $\vartheta_A^{-1}:\wp(\At(A))\to A$ which is given by $\vartheta_A^{-1} (S)= \bigvee S$ for $S \subseteq \At(A)$.
We  
set $\tau=\vartheta_A^{-1} \circ \wp(\gamma^*): \L(M) \to A$. 
\[
\begin{tikzcd}
\L(M)=\wp(M) \arrow[rr, bend left = 20, "\tau"] \arrow[r, "\wp(\gamma^*)"'] & \wp(\At(A)) \arrow[r, "\vartheta_A^{-1}"'] & A
\end{tikzcd}
\]
Then $\tau$ is the composition of two $\caba$-morphisms, so is a $\caba$-morphism. Moreover, for $S\subseteq M$, we have
\[
\tau(S)= \vartheta_A^{-1}\wp(\gamma^*)(S) = \vartheta_A^{-1}(\{ x \in \At(A) \mid \gamma^*(x) \in S \}) = \bigvee \{ x \in \At(A) \mid \gamma^*(x) \in S \}.
\]
Thus, for $a \in M$, we have
\[
\tau(\iota_M(a))=\tau(\down a) = \bigvee \{ x \in \At(A) \mid \gamma^*(x) \le a \} = \bigvee \{ x \in \At(A) \mid x \le \gamma(a) \} = \gamma(a) 
\]
since $A$ is atomic.
To show that $\tau$ is uniquely determined by the equation $\tau \circ \iota_M=\gamma$, it is enough to show that $\L(M)$ is generated as a complete boolean algebra by $\iota_M[M]$. Since each $S \subseteq M$ is the union of singletons, this follows from the equation 
\[
\{ a \} = \down a \setminus \{ b \mid b < a \} = \down a \setminus \bigcup \{ \down b \mid b < a \} = \iota_M(a) \wedge \neg \bigvee \{ \iota_M(b) \mid b < a \}.
\qedhere %%%%% put in \qedhere
\]
\end{proof}

\begin{rem}\label{rem:H on morphisms}
It is worth mentioning that the equation $\{ a \} = \iota_M(a) \wedge \neg \bigvee \{ \iota_M(b) \mid b < a \}$ above
allows an alternate description of $\tau:\L(M)\to A$ that does not involve atoms. 
Indeed, since $\tau$ is a $\caba$-morphism, we have
\begin{align*}
\tau(\{ a \}) &= \tau \left(\iota_M(a) \wedge \neg \bigvee \{ \iota_M(b) \mid b < a \} \right) = \tau \iota_M(a) \wedge \neg \bigvee \{ \tau \iota_M(b) \mid b < a \}\\
&= \gamma(a) \wedge \neg \bigvee \{ \gamma(b) \mid b < a \}.
\end{align*}
Thus, for each $S \subseteq M$, we have
\[
  \tau(S) = \bigvee \{ \tau(\{ a \}) \mid a \in S \} = \bigvee \left\lbrace \gamma(a) \wedge \neg \bigvee \{\gamma(b) \mid b < a \} \;\middle|\; a \in S \right\rbrace.
\]
\end{rem}

\begin{rem} \label{rem:morphisms}
To describe how $\L$ acts on morphisms, let $\gamma : M \to N$ be a $\csl$-morphism. Then $\iota_N \circ \gamma : M \to \L(N)$ is a $\csl$-morphism, so there is a unique $\caba$-morphism $\L(\gamma) : \L(M) \to \L(N)$ such that $\L(\gamma) \circ \iota_M = \iota_N \circ \gamma$.
\[
\begin{tikzcd}[column sep = 5pc]
M \arrow[r, "\gamma"] \arrow[d, "\iota_M"'] & N \arrow[d, "\iota_N"] \\
\L(M) \arrow[r, "\L(\gamma)"'] & \L(N)
\end{tikzcd}
\]
Therefore, if $a \in M$, then $\L(\gamma)(\down a) = \L(\gamma)\iota_M(a) = \iota_N\gamma(a) = \down \gamma(a)$.
\end{rem}

\section{Coalgebraic approach to Thomason duality}\label{sec:thomason}

\begin{defi}
Let $\H:\caba\to\caba$ be the composition $\H=\L\U$.
\[
\begin{tikzcd}[column sep = 5pc]
\caba \arrow[rr, bend left = 15, "\H"] \arrow[r, "\U"'] & \csl \arrow[r, "\L"'] & \caba
\end{tikzcd}
\]
\end{defi}

\begin{rem}
In Theorems~\ref{prop: H is a reflector Venema def} and~\ref{thm:alternate L} we have given two alternative constructions of $\L:\csl\to\caba$. Thus, we have two alternative descriptions of $\H:\caba\to\caba$. For $A \in \caba$ we can think of $\H(A)$ as the powerset of $A$ (Theorem~\ref{thm:alternate L}) or as the quotient of the free object in $\caba$ over $A$ (Theorem~\ref{prop: H is a reflector Venema def}). The resulting two functors are naturally isomorphic. In this section we will always assume that $\H(A)$ is the powerset of $A$, but will indicate how the corresponding result can be proved if we think of $\H(A)$ as the quotient of the free object in $\caba$ over $A$. 
\end{rem}

We show that the diagram in Figure~\ref{diagram} is commutative up to natural isomorphism. The horizontal arrows in the diagram represent the contravariant functors $\At:\Set \to^\op \caba$ and $\wp:\Set\to^\op\caba$ of Tarski duality, whereas the vertical arrows represent the endofunctors $\H$ and $\mathcal{P}$ on $\caba$ and $\Set$, respectively. This together with standard algebra/coalgebra machinery then allows us to prove that $\Alg(\H)$ is dually equivalent to $\Coalg(\mathcal{P})$, thus yielding an alternate proof of Thomason duality.
\begin{figure}[H]
\begin{tikzcd}[column sep = 5pc] 
\caba \arrow[r, shift left = .5ex, "\At"] \arrow[d, "\H"'] &  \Set \arrow[d, "\mathcal{P}"] \arrow[l, shift left, "\wp"] \\
\caba  \arrow[r, shift left = .5ex, "\At"] & \Set \arrow[l, shift left, "\wp"]
\end{tikzcd}
\caption{}\label{diagram}
\end{figure}

\begin{thm}\label{thm:Hwp=wpP}
  \hspace{1em}
\begin{enumerate}[$(1)$]
\item $\H \circ \wp = \wp\circ\mathcal{P}$.
\item $\At\circ\H$ is naturally isomorphic to $\mathcal{P}\circ\At$.
\end{enumerate}
\end{thm}

\begin{proof}
(1) If $X \in \Set$, then $\wp \mathcal{P}(X)$ and $\H \wp (X)$ are both objects in $\caba$ obtained by taking the double powerset of $X$ ordered by inclusion. We show that the two compositions also agree on morphisms. 
Let $f:X \to Y$ be a map. It is sufficient to show that $\wp \mathcal{P}(f)(\down S)=\H \wp(f)(\down S)$ for each $S \subseteq Y$. By Remark~\ref{rem:morphisms}, we have 
\[
\H \wp(f)(\down S) = \down \wp(f)(S)=\down f^{-1}(S).
\]
On the other hand, 
\begin{align*}
\wp \mathcal{P}(f)(\down S) &= \mathcal{P}(f)^{-1}(\down S)=\{ T \in \wp(X) \mid \mathcal{P}(f)(T) \in \down S \}\\
&=\{ T \in \wp(X) \mid f[T] \subseteq S \}=\{ T \in \wp(X) \mid T \subseteq f^{-1}(S) \}=\down f^{-1}(S).
\end{align*}
Thus, $\wp \mathcal{P}(f)(\down S)=\H \wp(f)(\down S)$, completing the proof.

(2) follows from (1) since the horizontal arrows in the diagram in Figure~\ref{diagram} form a dual equivalence.
\end{proof}

\begin{rem}\label{rem:box vs down}
If we use the alternative description of $\H$, then Theorem~\ref{thm:Hwp=wpP}(1) should be phrased as $\H \circ \wp$ is naturally isomorphic to $\wp\circ\mathcal{P}$. The natural isomorphism $\xi:\H \circ \wp\to\wp\circ\mathcal{P}$ is given on the generators of $\H\wp(X)$ by $\xi_X(\Box_S)=\down S$ for each $S\subseteq X\in\Set$.
\end{rem}

In the next remark we give an explicit description of the natural isomorphism $\zeta : \At\circ\H \to \mathcal{P}\circ\At$. This will be used in Remark~\ref{rem: 4.11}.

\begin{rem}\label{cor:AtH and PAt naturally iso}
For $A \in \caba$ define $\zeta_A : \At\H(A) \to \mathcal{P}\At(A)$ by 
\[
\zeta_A(\{ a \})= \{ x \in \At(A) \mid x \le a \}
\] 
for each $a \in A$. 
Since $\H \circ \wp = \wp\circ\mathcal{P}$ and $\varepsilon,\vartheta$ are natural isomorphisms of Tarski duality, we have that the composition $\At \H(\vartheta_A) \circ \varepsilon_{\mathcal{P} \At(A)}$ is a bijection. 
\[
\begin{tikzcd}[column sep = 5pc]
\mathcal{P} \At(A) \arrow[r, "\varepsilon_{\mathcal{P} \At(A)}"] & \At \wp \mathcal{P} \At(A) = \At \H \wp \At(A) \arrow[r, "\At \H(\vartheta_A)"] & \At \H(A)
\end{tikzcd}
\]
We show that for each $a \in A$ we have 
\[
(\At \H(\vartheta_A) \circ \varepsilon_{\mathcal{P} \At(A)})(\{ x \in \At(A) \mid x \le a \})=\{a \}.
\]
Since $\varepsilon_{\mathcal{P} \At(A)}(\{ x \in \At(A) \mid x \le a \})=\{ \{ x \in \At(A) \mid x \le a \} \}$, it is sufficient to prove that
\[
\At \H(\vartheta_A)(\{ \{ x \in \At(A) \mid x \le a \} \})=\{a \}.
\]
It follows from Remark~\ref{rem:H on morphisms} that
\begin{align*}
\H(\vartheta_A)(\{a \}) &=\down \vartheta_A(a) \wedge \neg \bigvee \{ \down \vartheta_A(b) \mid b < a \} \\
&=\down \{ x \in \At(A) \mid x \le a \} \setminus \bigcup \{ \down \{ x \in \At(A) \mid x \le b \} \mid b < a \} \\
&=\{ \{ x \in \At(A) \mid x \le a \} \}.
\end{align*}
In particular, $\{ \{ x \in \At(A) \mid x \le a \} \} \le \H(\vartheta_A)(\{ a \})$, and so 
\[
\At \H(\vartheta_A) (\{ \{ x \in \At(A) \mid x \le a \} \}) \le \{ a \}
\]
 because $\At\H(\vartheta_A)$ is left adjoint to $\H(\vartheta_A)$. Therefore, $\At \H(\vartheta_A)(\{ \{ x \in \At(A) \mid x \le a \} \})=\{a \}$ since both sides of the last inequality are atoms. Thus, $\zeta_A =(\At \H(\vartheta_A) \circ \varepsilon_{\mathcal{P} \At(A)})^{-1}$, and hence $\zeta$ is a natural isomorphism.
\end{rem}

We next utilize Theorem~\ref{thm:Hwp=wpP} and standard algebra/coalgebra machinery to show that Tarski duality lifts to a dual equivalence between $\Alg(\H)$ and $\Coalg(\mathcal P)$. 
We start with the following well-known result (see, e.g.,~\cite[Sec.~9]{Ven07}). Since we will be using the functors establishing the isomorphism of Theorem~\ref{thm: KFr} in Remark~\ref{rem: canonical ext on algebras}, we sketch the proof.

\begin{thm}\label{thm: KFr}
$\KFr$ is isomorphic to $\Coalg(\mathcal{P})$.
\end{thm}

\begin{proof}[Proof (Sketch)]
To each Kripke frame $\mathfrak F=(X,R)$ we associate the coalgebra $\rho_R:X \to \mathcal{P}(X)$ defined by $\rho_R(x)=R[x]$. If $f:X_1 \to X_2$ is a p-morphism between Kripke frames $(X_1, R_1)$ and $(X_2, R_2)$, then $f$ is also a morphism between the coalgebras $(X_1,\rho_{R_1})$ and $(X_2, \rho_{R_2})$. This defines a covariant functor $\C : \KFr \to \Coalg(\mathcal{P})$. To each coalgebra $(X,\rho)$ for $\mathcal{P}$, we associate the Kripke frame $(X,R_\rho)$ where $xR_\rho y$ iff $y \in \rho(x)$. If $f$ is a morphism between two coalgebras $(X_1,\rho_1)$ and $(X_2,\rho_2)$, then $f$ is also a p-morphism between the Kripke frames $(X_1,R_{\rho_1})$ and $(X_2, R_{\rho_2})$. This defines a covariant functor $\F : \Coalg(\mathcal{P}) \to \KFr$. It is straightforward to see that $R=R_{\rho_R}$ for each $(X,R) \in \KFr$ and $\rho=\rho_{R_\rho}$ for each $(X,\rho) \in \Coalg(\mathcal{P})$. Thus, the functors $\C$ and $\F$ yield an isomorphism of $\KFr$ and $\Coalg(\mathcal{P})$.
\end{proof}

We next show that $\cama$ is isomorphic to $\Alg(\H)$. This is parallel to the well-known fact that $\ma$ is isomorphic to $\Alg(\K)$ (see, e.g.,~\cite[Cor.~3.11]{KKV04}). 

\begin{thm}\label{thm: cama}
$\cama$ is isomorphic to $\Alg(\H)$.
\end{thm}

\begin{proof}
Let $(A, \Box) \in \cama$. 
Since $\Box : A \to A$ is a $\csl$-morphism, by Theorem~\ref{thm:alternate L}, there is a unique $\caba$-morphism $\tau_\Box : \H(A) \to A$ such that $\tau_\Box(\down a) = \Box a$ for each $a \in A$. Therefore, $(A, \tau_\Box) \in \Alg(\H)$. Let $\alpha : A_1 \to A_2$ be a $\cama$-morphism and $a \in A_1$. Since $\alpha(\Box_1 a)=\Box_2 \alpha(a)$, by Remark~\ref{rem:morphisms},
\[
\tau_{\Box_2} \H(\alpha)(\down a) = \tau_{\Box_2}(\down {\alpha(a)}) = \Box_2 {\alpha(a)} = \alpha(\Box_1 a) = \alpha\tau_{\Box_1}(\down a).
\]
Because $\H(A)$ is generated by $\{ \down a \mid a \in A\}$, we obtain that $\tau_{\Box_2} \circ \H(\alpha) = \alpha \circ \tau_{\Box_1}$. Therefore, $\alpha$ is also a morphism in $\Alg(\H)$. This defines a covariant functor $\A : \cama \to \Alg(\H)$.

Conversely, let $(A,\tau) \in \Alg(\H)$ so $A \in \caba$ and $\tau:\H(A) \to A$ is a $\caba$-morphism. If we define $\Box_\tau$ on $A$ by $\Box_\tau a = \tau(\down a)$, it is easy to see that $\Box_\tau$ is completely multiplicative, so $(A, \Box_\tau) \in \cama$. Let $\alpha : A_1 \to A_2$ be a morphism in $\Alg(\H)$ and $a \in A_1$. By Remark~\ref{rem:morphisms},
\[
\Box_{\tau_2} \alpha(a) = \tau_2(\down {\alpha(a)}) = \tau_2 \H(\alpha)(\down a) = \alpha\tau_1(\down a) = \alpha(\Box_{\tau_1} a).
\]
Therefore, $\alpha$ is also a $\cama$-morphism. This defines a covariant functor $\M :  \Alg(\H) \to \cama$.

Let $(A, \Box) \in \cama$. For $a \in A$, we have $\Box_{\tau_\Box}a = \tau_\Box(\down a) = \Box a$. Therefore, $\Box_{\tau_\Box} = \Box$.  Next, let $(A, \tau) \in \Alg(\H)$. For $a \in A$, we have $\tau_{\Box_\tau}(\down a) = \Box_\tau a = \tau(\down a)$. Since $\H(A)$ is generated by $\{ \down a \mid a \in A \}$, we obtain that $\tau_{\Box_\tau}  = \tau$. Thus, the functors $\A$ and $\M$ yield an isomorphism of $\cama$ and $\Alg(\H)$.
\end{proof}

\begin{rem}
If we use the alternative description of $\H$, then the previous theorem can be proved using Theorem~\ref{prop: H is a reflector Venema def} and Remark~\ref{rem: H on maps}. The advantage of using this description of $\H$ lies in the suggestive definitions $\tau_\Box(\Box_a) = \Box a$ and $\Box_\tau a = \tau(\Box_a)$.
\end{rem}

We are ready to lift Tarski duality to a dual equivalence between $\Alg(\H)$ and $\Coalg(\mathcal{P})$. For this we utilize 
\cite[Thm.~2.5.9]{Jac17} which states that, under certain conditions, adjunctions lift to adjunctions between categories of algebras. For our purposes, we require the following reformulation of 
\cite[Thm.~2.5.9]{Jac17} for dual equivalences.

\begin{lem}\label{lem: Jacobs}
Let $\mathbb{S}:\sf{C} \to \sf{C}$, $\mathbb{T}:\sf{D} \to \sf{D}$ be two endofunctors and $\mathbb{Q}:\sf{C} \to^\op \sf{D}$, $\mathbb{R}:\sf{D} \to^\op \sf{C}$ two contravariant functors forming a dual equivalence. Suppose that $\mathbb{S}\mathbb{R}$ and $\mathbb{R}\mathbb{T}$ are naturally isomorphic \emph{(}and hence so are $\mathbb{T}\mathbb{Q}$ and $\mathbb{Q}\mathbb{S}$\emph{)}. Then $\mathbb{Q}$ and $\mathbb{R}$ lift to contravariant functors
$\widehat{\mathbb{Q}}:\Alg(\mathbb{S}) \to^\op \Coalg(\mathbb{T})$ and $\widehat{\mathbb{R}}:\Coalg(\mathbb{T}) \to^\op \Alg(\mathbb{S})$ 
which yield a dual equivalence between $\Alg(\mathbb{S})$ and $\Coalg(\mathbb{T})$.
\[
\begin{tikzcd}[column sep = 5pc] 
\sf{C} \arrow[r,shift left = .5ex, "\mathbb{Q}"] \arrow[d, "\mathbb{S}"'] & \sf{D} \arrow[d, "\mathbb{T}"] \arrow[l, shift left = .5ex, "\mathbb{R}"] \\
\sf{C} \arrow[r,shift left = .5ex, "\mathbb{Q}"] & \sf{D} \arrow[l, shift left = .5ex, "\mathbb{R}"]
\end{tikzcd}
\]
\end{lem}

Theorem~\ref{thm:Hwp=wpP} and Lemma~\ref{lem: Jacobs} then immediately give:

\begin{thm}\label{cor: main}
Tarski duality between $\caba$ and $\Set$ lifts to a dual equivalence between $\Alg(\H)$ and $\Coalg(\mathcal{P})$.
\end{thm}

\begin{rem} \label{rem: 4.11}
By adapting the proof of~\cite[Thm.~2.5.9]{Jac17}, the contravariant functors $\widehat{\At} : \Alg(\H) \to^\op \Coalg(\mathcal{P})$ and $\widehat{\wp} :\Coalg(\mathcal{P}) \to^\op \Alg(\H)$ lifting $\At:\caba \to^\op \Set$ and $\wp:\Set \to^\op \caba$ can be defined as follows: 

If $(A, f) \in \Alg(\H)$, then 
$(\At(A), \zeta_A \circ \At(f))\in \Coalg(\mathcal{P})$ where $\zeta_A$ is defined in Remark~\ref{cor:AtH and PAt naturally iso}. Moreover, if $\alpha : A_1 \to A_2$ is a morphism in $\Alg(\H)$, then $\At(\alpha)$ is a morphism of the corresponding coalgebras. This defines the functor $\widehat{\At} : \Alg(\H) \to^\op \Coalg(\mathcal{P})$.
\[
\begin{tikzcd}[column sep = 5pc]
\At(A_2) \arrow[d, "\At(\alpha)"'] \arrow[r, "\zeta_{A_2} \circ \At(f_2)"] & \mathcal{P}\At(A_2) \arrow[d, "\mathcal{P}\At(\alpha)"]\\
\At(A_1) \arrow[r, "\zeta_{A_1} \circ \At(f_1)"'] & \mathcal{P}\At(A_1)
\end{tikzcd}
\]

Let $(X,g) \in \Coalg(\mathcal{P})$. Then 
$\wp(g): \wp\mathcal{P}(X) \to \wp(X)$ is a $\caba$-morphism. Since $\H \circ \wp = \wp\circ\mathcal{P}$ (see Theorem~\ref{thm:Hwp=wpP}(1)), we have that $(\wp(X), \wp(g)) \in \Alg(\H)$. Moreover, if $h:(X_1,g_1) \to (X_2,g_2)$ is a $\Coalg(\mathcal{P})$-morphism, then $\wp(h)$ is an $\Alg(\H)$-morphism. This defines the functor $\widehat{\wp}:\Coalg(\mathcal{P}) \to^\op \Alg(\H)$.
\[
\begin{tikzcd}[column sep = 5pc]
\H\wp(X_2)=\wp \mathcal{P}(X_2) \arrow[r, "\wp(g_2)"] \arrow[d, shift left=12, "\wp \mathcal{P}(h)"] \arrow[d, shift right=7, "\H \wp(h)"'] & \wp(X_2) \arrow[d, "\wp(h)"] \\
\H\wp(X_1)=\wp \mathcal{P}(X_1) \arrow[r, "\wp(g_1)"'] & \wp(X_1)
\end{tikzcd}
\]
\end{rem}

\begin{rem}
Putting Theorems~\ref{thm: KFr}, \ref{thm: cama}, and \ref{cor: main} together yields an alternate proof of Thomason duality. 
We recall that the contravariant functors establishing Thomason duality extend the contravariant functors of Tarski duality. Namely, the functor $\wp:\KFr\to^\op\cama$ associates to each $(X,R) \in \KFr$ the algebra $(\wp(X), \Box_R) \in \cama$ where $\Box_R$ is defined by
\[
  \Box_R(S)=\{ x \in X \mid R[x] \subseteq S \}.\footnote{Thus, $\Box_R$ is the predicate lifting for $\Box$ composed with the coalgebra map $\rho_R : X \to \mathcal{P}(X)$.}
\] 
Also, $\wp$ associates to each $\KFr$-morphism $f:X\to Y$ the $\cama$-morphism $\wp(f):\wp(Y)\to\wp(X)$ given by $\wp(f)=f^{-1}$. 
The functor $\At:\cama\to^\op\KFr$ associates to each $(A,\Box) \in \cama$ the Kripke frame $(\At(A),R_\Box)$ where 
\[
xR_\Box y \mbox{ iff } x \wedge \Box\neg y=0 \mbox{ iff } (\forall a \in A) (x \le \Box a \Rightarrow y \le a).
\] 
Also, $\At$ associates to each $\cama$-morphism $\alpha:A\to B$ the p-morphism $\At(\alpha) : \At(B)\to\At(A)$ given by $\At(\alpha)=\alpha^*$. 

We show that $\wp : \KFr \to^\op \cama$ is the composition
\[
\begin{tikzcd}
\KFr \arrow[r, "\C"] & \Coalg(\mathcal{P}) \arrow[r, "\widehat{\wp}"] & \Alg(\H) \arrow[r, "\M"] & \cama
\end{tikzcd}
\]
and that $\At : \cama \to^\op \KFr$ is the composition
\[
 \begin{tikzcd}
 \cama \arrow[r, "\A"] & \Alg(\H) \arrow[r, "\widehat{\At}"] & \Coalg(\mathcal{P}) \arrow[r, "\F"] & \KFr.
 \end{tikzcd}
 \]

Let $(X,R) \in \KFr$. Then 
\[
\M \, \widehat{\wp} \, \C(X,R) = \M \, \widehat{\wp} \, (X, \rho_R) =\M (\wp(X), \wp(\rho_R))=(\wp(X), \Box_{\wp(\rho_R)}).
\]
For $S \subseteq X$, we have
\[
\Box_{\wp(\rho_R)} S = \wp(\rho_R)(\down S)=\rho_R^{-1}({\downarrow} S)=\{ x \in X \mid \rho_R(x) \subseteq S \} = \{ x \in X \mid R[x] \subseteq S \} = \Box_R S.
\]
Thus, $\Box_{\wp(\rho_R)}=\Box_R$, so $\M \, \widehat{\wp} \, \C(X,R) = (\wp(X), \Box_R)$, and hence $\M \, \widehat{\wp} \, \C = \wp$.

Let $(A,\Box)\in\cama$ . Then 
\[
\F \, \widehat{\At} \, \A(A, \Box) = \F \, \widehat{\At} \, (A, \tau_\Box)= \F (\At(A), \zeta_{A} \circ \At(\tau_\Box)) =(\At(A), R_{\zeta_{A} \circ \At(\tau_\Box)})
\]
For $x,y \in \At(A)$, we have 
\[
x R_{\zeta_{A} \circ \At(\tau_\Box)} y \ \mbox{ iff } \ y \in \zeta_{A}\At(\tau_\Box)(x) \ \mbox{ iff } \ y \in \zeta_A((\tau_\Box)^*(x)).
\]
By Remark~\ref{rem:H on morphisms}, for $a \in A$, we have $\tau_\Box(\{ a \})= \Box a \wedge \neg \bigvee \{ \Box b \mid b < a \}$. 
Therefore, for each $x \in \At(A)$, we have 
\[
x \le \tau_\Box(\{ a \}) \mbox{ iff } x \le \Box a \mbox{ and } x \nleq \Box b \mbox{ for each } b < a.
\] 
Since $\Box$ is completely multiplicative, $c:=\bigwedge \{ a \in A \mid x \le \Box a \}$ is the least element satisfying $x \le \Box c$. Thus, $x \le \tau_\Box(\{ c \})$. Consequently, $(\tau_\Box)^*(x)=\{ c \}$
since $(\tau_\Box)^*$ is left adjoint to $\tau_\Box$ and both $(\tau_\Box)^*(x)$ and 
$\{ c \}$
 are atoms of $\H(A)$.
It follows that
\[
y \in \zeta_A((\tau_\Box)^*(x)) \ \mbox{iff} \ y \le \bigwedge \{ a \in A \mid x \le \Box a \} \ \mbox{iff} \ (\forall a \in A) (x \le \Box a \Rightarrow y \le a).
\]
Thus, $x R_{\zeta_{A} \circ \At(\tau_\Box)} y$ iff $x R_\Box y$, so $\F \, \widehat{\At} \, \A(A, \Box) = \At(A,\Box)$, and hence $\F \, \widehat{\At} \, \A = \At$.
\end{rem} 

\begin{rem}\label{rem: canonical ext on algebras}
We conclude the paper by connecting the coalgebraic approaches to J\'onsson-Tarski and Thomason dualities. As follows from~\cite{KKV04}, $\Alg(\K)$ is dually equivalent to $\Coalg(\V)$, from which J\'onsson-Tarski duality follows. By Theorem~\ref{cor: main}, $\Alg(\H)$ is dually equivalent to $\Coalg(\mathcal P)$, from which Thomason duality follows. 

Let $\U:\Stone \to \Set$ be the forgetful functor. For each $X \in \Stone$, viewing the underlying set of the Vietoris space $\V(X)$ as a subset of $\mathcal P(X)$, we have an inclusion map $i : \U\V(X) \to \mathcal{P}\U(X)$. We extend $\U$ to a forgetful functor on the level of coalgebras. Let $(X,g) \in \Coalg(\V)$, so $g:X \to \V(X)$ is a continuous map. Set $\U(X,g):=(\U(X), g')$ where $g' : \mathcal{U}(X) \to \mathcal{PU}(X)$ is given by $g'=i \circ \U(g)$.
\[
\begin{tikzcd}
\U(X) \arrow[rr, bend right = 18, "g'"'] \arrow[r, "\U(g)"] & \U\V(X) \arrow[r, "i"] &  \mathcal{P}\U(X)
\end{tikzcd}
\]
Then $(\U(X), g') \in \Coalg(\mathcal{P})$. If $\alpha:(X_1,g_1)\to(X_2,g_2)$ is a $\Coalg(\V)$-morphism, it is straightforward to see that $\U(\alpha):(\U(X_1),g_1')\to(\U(X_2),g_2')$ is a $\Coalg(\mathcal{P})$-morphism. This yields a functor $\U : \Coalg(\V) \to \Coalg(\mathcal{P})$. If we identify $X$ with $\U(X)$ and $g : X \to \V(X)$ with $g' : \U(X) \to \wp\U(X)$, then $\U(X,g) = (X,g)$. We have thus forgotten the topological structure of $X$ and the continuity of $g$. This justifies thinking about $\U : \Coalg(\V) \to \Coalg(\mathcal{P})$ as a forgetful functor.

We next define a functor $(-)^\sigma : \Alg(\K) \to \Alg(\H)$ which can be thought of as the canonical extension functor for algebras for $\K$. Use Lemma~\ref{lem: Jacobs} to lift the contravariant functors of Stone duality to $\widehat{\uf},\widehat{\Clop}$ and those of Tarski duality to $\widehat{\At},\widehat{\wp}$ and set $(-)^\sigma = \widehat{\wp} \circ \U \circ \widehat{\uf}$.
Thus, $(-)^\sigma \circ \widehat{\Clop}$ and $\widehat{\wp}\circ \U$ are naturally isomorphic, and hence so are $\U\circ \widehat{\uf}$ and $\widehat{\At} \circ (-)^\sigma$.

\[
\begin{tikzcd}[column sep = 5pc]
\Alg(\K) \arrow[r, shift right = .5ex, "\widehat{\uf}"'] \arrow[d, "(-)^\sigma"'] & \Coalg(\V) \arrow[d, "\U"] \arrow[l, shift right = .5ex, "\widehat{\Clop}"'] \\
\Alg(\H) \arrow[r, shift right = .5ex, "\widehat{\At}"'] & \Coalg(\mathcal{P}) \arrow[l, shift right = .5ex, "\widehat{\wp}"']
\end{tikzcd}
\]

The functor $\K$ satisfies $\uf\K(A) \cong \V\uf(A)$~\cite[Cor.~3.11]{KKV04}. By identifying these two Stone spaces, for $(A, \alpha) \in \Alg(\K)$ we describe $(A, \alpha)^\sigma$ directly. Since $\alpha : \K(A) \to A$ is a $\ba$-morphism, $\uf(\alpha) : \uf(A) \to \uf\K(A) = \V\uf(A)$ is a continuous map. Therefore, $(\uf(A), \uf(\alpha)) \in \Coalg(\V)$. The forgetful functor (after appropriate identifications) sends this to $(\uf(A), \uf(\alpha)) \in \Coalg(\mathcal{P})$. Finally, $\widehat{\wp}$ sends $(\uf(A), \uf(\alpha))$ to $(\wp\uf(A), \wp\uf(\alpha)) = (A^\sigma, \alpha^\sigma)$. Because of these calculations, we can view the functor $\Alg(\K) \to \Alg(\H)$ as the canonical extension functor for algebras for $\K$.
\end{rem}

\section*{Acknowledgements}

We would like to thank the referees for their comments. One of the referees suggested to us an alternative description of $\L$ given in Theorem~\ref{thm:alternate L}. We also thank Nick Bezhanishvili, Jim de Groot, Sebastian Enqvist, Clemens Kupke, Phil Scott, and Yde Venema for their comments on an earlier version of the paper.

\bibliographystyle{alphaurl}
\bibliography{../Gelfand}

\end{document}